\documentclass[12pt]{amsart}
\usepackage[lite]{amsrefs}

\usepackage{amssymb}

\usepackage[all,cmtip]{xy}
\usepackage[margin= 1.5in]{geometry}

\usepackage{comment}

\newcommand{\zprim}{\mathbb{Z}_\text{prim}^s}

\newcommand{\R}{\mathbb{R}}

\newcommand{\vecx}{\bold{x}}
\newcommand{\vecv}{\bold{v}}

\newcommand{\vecy}{\bold{y}}
\newcommand{\vecw}{\bold{w}}
\newcommand{\vecz}{\bold{z}}
\newcommand{\vecb}{\bold{b}}
\newcommand{\vecu}{\bold{u}}
\newcommand{\veca}{\bold{a}}
\newcommand{\vect}{\bold{t}}
\newcommand{\eps}{\varepsilon}

\numberwithin{equation}{section}

\theoremstyle{plain} 
\theoremstyle{break}
\theoremstyle{break}

\newtheorem{thm}{Theorem}[section]

\newtheorem{lem}[thm]{Lemma}

\begin{document}

\title[Manin's conjecture for $x_1y_1^2 + \ldots + x_sy_s^2 =0$]{Density of Rational Points on diagonal bidegree $(1,2)$ hypersurface in $\mathbb{P}^{s-1} \times \mathbb{P}^{s-1}$}

\author{Xun Wang}

\address{}




\date{}

\begin{abstract}
In this paper we establish an asymptotic formula for the number of rational points of bounded anticanonical height which lie on a certain Zariski dense subset of the biprojective hypersurface
\begin{align*}
        x_1y_1^2+...+x_sy_s^2 = 0
\end{align*}
in $\mathbb{P}^{s-1} \times \mathbb{P}^{s-1} $ with $s \geq 7$. This confirms the Manin conjecture for this variety.
\end{abstract}

\maketitle

\tableofcontents

\section{Introduction}

Consider the biprojective hypersurface $V_s$ in $\mathbb{P}^{s-1} \times \mathbb{P}^{s-1}$ defined by the equation
\[ x_1 y_1^2 + \ldots + x_s y_s^2 =0.\] The purpose of this article is to establish Manin's conjecture for such varieties with $s \geq 7$.
For a point $(x,y) \in V_s(\mathbb{Q})$ represented by a vector $(\mathbf{x}, \mathbf{y}) \in \mathbb{Z}^s_{\text{prim}} \times \mathbb{Z}^s_{\text{prim}}$, the anticanonical height function is given by $H(x,y) = \vert \mathbf{x} \vert^{s-1} \vert \mathbf{y} \vert^{s-2}$, with $\vert \cdot \vert: \mathbb{R}^s \rightarrow \mathbb{R}$ the usual sup-norm.
\begin{thm}
If $s\geq 7$ then there exists an open subset $U \subset V_s$ such that 
\[
\#\{ (x,y) \in U(\mathbb{Q}) : H(x,y) \leq B\}
\sim c_{\text{Peyre}} B \log B,
\] as $B \rightarrow \infty$.
\end{thm}
Note that the leading constant arising in the theorem is as predicted by Peyre~\cite{Peyre} and that, since $\text{Pic}(V) \cong \mathbb{Z}^2$, the order of magnitude agrees with the prediction of Manin~\cite{FMT}.

In \cite{Schindler}, Schindler proves Manin's conjecture for any biprojective hypersurface of bidegree $(d_1, d_2)$ in $\mathbb{P}^{n_1-1} \times \mathbb{P}^{n_2-2}$ for $d_1, d_2 \geq 2$ and $n_1, n_2$ sufficiently large. Due to the assumption on the degrees her results do not cover the case in which we are interested. In general, the literature is lacking an equivalent result to Schindler's which will handle the bidegree $(1,d)$ setting for sufficiently many variables.
Diagonal equations have long been an important testing ground in the study of rational points and there has been a great deal of study of Manin's conjecture for diagonal hypersurfaces of the form 
\begin{equation}\label{d1d2}
x_1^{d_1}y_1^{d_2} + \ldots + x_s^{d_1}y_s^{d_2}=0.
\end{equation} 
If $d_1 = d_2 = 1$ then Manin's conjecture holds for $n \geq 2$ and this has been shown in multiple different ways: the work of Franke--Manin--Tschinkel~\cite{FMT} proves this using the theory of related Eisenstein series, Thunder~\cite{Thunder} provides a proof using the geometry of numbers and Robbiani~\cite{Rob} and Spencer~\cite{Spencer} prove Manin's conjecture via the circle method. In the case,  $d_1 =1, d_2 =2$ (the case in which the present paper is interested) upper and lower bounds of the correct order of magnitude where established for $s=3$ by Le Boudec~\cite{LB}. Recently, the $s=4$ case was studied by Browning--Heath-Brown~\cite{Browning2}. They showed that the expected asymptotic growth holds under the condition that one removes a thin subset. In particular, one needs to avoid counting those points $(\mathbf{x}, \mathbf{y})$ such that $x_1x_2x_3x_4 \in \mathbb{Q}^2$, which is an infinite union of closed subsets. This thin set phenomenon was in fact first observed, by Batyrev--Tschinkel~\cite{BT}, for varieties of the shape \eqref{d1d2} with $d_1 =1, d_2 = 3$. Our goal is to expand the literature by considering the $(1,2)$ setting for higher values of $s$. We note that in our result, one does not need to worry about the removal of a thin set, as a closed subset will suffice. However, it is certainly necessary to remove a closed subset as , for instance, the points such that $x_1 = 0$ and $y_2 = \ldots = y_s = 0$ will contribute $\gg B^s$ to the count.

Our approach is very similar to that of Browning--Heath-Brown. We proceed in two stages: for small $\mathbf{y}$ we may fix the $\mathbf{y}$ and consider the resulting linear equation in $\mathbf{x}$. We will study the number of solutions to this equation using techniques from the geometry of numbers. When $\mathbf{x}$ is small, we fix these variables and consider the resulting quadratic equation. This we can attack using the circle method. It is in this step that we diverge from the work of Browning--Heath-Brown. In \cite{Browning2}, the authors apply the $\delta$-method in the form introduced by Heath-Brown~\cite{HB} since they are dealing with a quadratic equation in 4 variables. For quadratic equations in at least 5 variables, the traditional form of the circle method is effective and so it is this form which we have chosen to use. One may consider the equation which we aim to solve as being very similar to studying Waring's problem, except with coefficients, a problem which is well understood classically (see \cite[Chapter 8]{Davenport}). However, we need all our results to be uniform in the coefficients as we need to sum over them after, and this addition to the classical problem is novel. As such we need to develop uniform versions of the key ingredients in the study of Waring's problem (such as Weyl's inequality, Hua's lemma, etc). Much of this uniform machinery could be used later to study varieties of a similar shape.

\textbf{Acknowledgements.} During the preparation of this paper the author is supported by NSF under Grant No. DMS-101601, while working at University of Michigan as part of the Math REU program. The author would like to thank Nicholas Rome and Will Dudarov for the numerous helpful comments and University of Michigan Math department for this opportunity. \\
\indent \textbf{Notation} Throughout this paper $\eta$ and $\eps$ will denote arbitrarily small constants, thus $c \eps$ will thus be rewritten as $\eps$ and similar cases hold for other expressions depending on $\eta$ and $\eps$. The constants $c_i$, $d$ and $\delta$ with $i \in \mathbb{N}$, whenever used, will denote certain small positive constants that's difficult to keep track of but easy to verify the value of. 

\section{Asymptotic formula using lattice counting}
This section is devoted to solving the equation when $|\vecy|$ is small. We follow closely with the work of Browning and Heath-Brown(See \cite[\S 3]{Browning2}) and arrive at various results using counting arguments from the geometry of numbers. \\
\indent Denote
\begin{align*}
    M_1(\vecy; R) = \#\{\vecx \in \mathbb{Z}^s: |\vecx| \leq R, \text{ }F(\vecx ;\vecy) = 0 \}
\end{align*}
for fixed $\vecy \in \zprim$. Define
\begin{align*}
    N_1(Y; B) = \# \Bigg\{ (\vecx,\vecy) &\in \mathbb{Z}^s \times \zprim: \begin{aligned}
        F&(\vecx;\vecy) = 0,\\|\vecx|^{s-1}|\vecy|^{s-2} &\leq B, Y \leq |\vecy| <  2Y
    \end{aligned} \Bigg\}.
\end{align*}
Then from our construction one sees that
\begin{align*}
    N_1(Y; B) = \sum_{\substack{\vecy \in \zprim\\ Y \leq |\vecy| < 2Y}}M_1 \big(\vecy;{(B/|y|^{s-2})}^{1/(s-1)} \big).
\end{align*}
The rest of this section is devoted to proving asymptotic formula for $N_1(Y;B)$; this is made explicit in the next theorem. 
\begin{thm}
Let $Y \geq 1$. Then
\begin{align}
\label{latres}
    N_1(Y;B) = \sum_{\substack{\vecy \in \zprim \\ Y \leq |\vecy| < 2Y}}\frac{\varrho_\infty(\vecy)}{|\vecy|^{s-2}}+ O(Y^s)+O(B^{(s-2)/(s-1)}Y^{(s+2)/(s-1)}),
\end{align}
where 
\begin{equation}
    \label{integral}
    \varrho_\infty(\vecy) = \int_{-\infty}^\infty \int_{[-1,1]^s}e(-\theta F(\vecx;\vecy))d\vecx d\theta   .
\end{equation}
\end{thm}
We note that if $\vecy \in \zprim$ has at least 2 nonzero components then
\begin{align*}
    \int_{[-1,1]^s}e(-\theta F(\vecx;\vecy))d\vecx = \prod_{j \leq s}\frac{\sin(2 \pi \theta y_j^2)}{\pi \theta y_j^2} \ll_\vecy (1+|\theta|)^{-2},
\end{align*}
and so the outer integral is absolutely convergent. If $\vecy = (1,0,...)$ then 
\begin{align*}
     \int_{[-1,1]^s}e(-\theta F(\vecx;\vecy))d\vecx = 2^{s-1}\frac{\sin(2\pi \theta)}{\pi \theta}
\end{align*}
and the outside integral is conditionally convergent with value $2^{s-1}$.\\
\indent We start by estimating $M_1(\vecy;R)$ for individual $\vecy \in \zprim$.
\begin{lem}
Let $\vecy \in \zprim $ and let $d(\vecy) = \sqrt{y_1^2+...+y_s^2}$. Let $V(\vecy)$ be the volume of the intersection of unit cube $[-1,1]^s$ and the hyperplane
\begin{align*}
    \{\vecx \in \mathbb{R}^s: F(\vecx;\vecy) = 0\}.
\end{align*}
Then there exists constant $c > 0$ and $\vecx_1 \in \zprim$ satisfying
\begin{align*}
    0 < |\vecx_1| < c|\vecy|^{2/(s-1)} \text{ and } F(\vecx_1; \vecy) = 0
\end{align*}
such that 
\begin{align*}
    M_1(\vecy;R) = \frac{V(\vecy)}{d(\vecy)}R^{s-1}+ O\bigg(\frac{R^{s-2}}{|\vecx_1|^{s-2}}\bigg)+O(1).
\end{align*}
\end{lem}
\begin{proof}
We first note that $M_1(\vecy; R)$ counts the number of vectors $\vecx \in \mathbb{Z}^s$ in an $s-1$ dimensional lattice of determinant $d(\vecy)$(c.f. \cite[Lemma 1(i)]{Schmidt}). \cite[Lemma 2]{Schmidt} then implies
\begin{align*}
    M_1(\vecy;R) = \frac{V(\vecy)}{d(\vecy)}R^{s-1} + O\bigg(\sum_{i=1}^{s-2}\frac{R^{i}}{|\vecx_1|...|\vecx_i|}\bigg)+O(1).
\end{align*}
after replacing the $L^2$ norm with $L^\infty$ norm; this replacement of norm is addressed in Appendix A. In particular, in Schmidt's notation $\vecx_i$ corresponds to the $i$-th successive minima. It follows that $R^{i}/(|\vecx_1|...|\vecx_i|) \ll R^i/|\vecx_1|^i \ll \min(1, R^{s-2}/|\vecx_1|^{s-2})$ for all $i \leq s-2$. Finally we see that 
\begin{align*}
    |\vecx_1| \leq (|\vecx_1|..|\vecx_{s-1}|)^{1/(s-1)} \ll d(\vecy)^{2/(s-1)} \ll |\vecy|^{2/(s-1)}.
\end{align*}
Lemma 2.2 hence follows. 
\end{proof}
An immediate consequence of Lemma 2.2 is that 
\begin{equation}
\label{partialres}
     N_1(Y;B) = B \sum_{\substack{\vecy \in \zprim\\ Y \leq |\vecy| < 2Y}}\frac{V(\vecy)}{|\vecy|^{s-2}d(\vecy)}+O(B^{(s-2)/(s-1)}|\vecy|^{-{(s-2)}^2/(s-1)} \Sigma_1)+O(Y^s),
\end{equation}
where 
\begin{align*}
    \Sigma_1 = \sum_{\substack{\vecy \in \zprim\\ Y \leq |\vecy| <  2Y}}\frac{1}{|\vecx_1|^{s-2}}.
\end{align*}
Consider the following sum
\begin{align*}
    E(Y) =  \sum_{\substack{\vecx \in \zprim\\0 < |\vecx| \leq cY^{2/(s-1)}}}\sum_{\substack{\vecy \in \zprim\\ Y \leq |\vecy| < 2Y\\F(\vecx;\vecy) = 0 }} \frac{1}{|\vecx|^{s-2}}.
\end{align*}
Clearly $\Sigma_1 \leq E(Y)$. We shall prove an upper bound for $E(Y)$, which shows that $\Sigma_1$ makes a satisfactory contribution to the error term appeared in Theorem 2.1.
\begin{lem} $E(Y) \ll Y^{4/(s-1)+s-2}$ for all $Y \geq 1$.
\end{lem}
\begin{proof}
We divide the first sum in accordance with the ranges of $|\vecx|$ into dyadic intervals, and note that the constant $c$ in this case can be ignored for our error term.
It's then immediate that 
\begin{align*}
    E(Y) \ll \sum_{L}L^{-s+2} \#\bigg\{(\vecx,\vecy) \in \zprim \times \zprim: \begin{aligned}
         F&(\vecx;\vecy) = 0\\
         L/2 \leq |\vecx| < &L, \text{ } Y \leq |\vecy| < 2Y
    \end{aligned} \bigg\},
\end{align*}
where summation is over $L \ll Y^{2/(s-1)}$ with L being powers of $2$ only. From \cite[Theorem 2.1]{Browning1} we obtain that
\begin{align*}
    E(Y) &\ll \sum_{L}L^{-s+2} \sum_{|\vecx| \sim L/2}\big( Y^{s-2}+L^{3s/4}Y^{(s-2)/2} \big)    \\
    &\ll \sum_{L}L^{2} Y^{s-2} \ll Y^{4/(s-1)+s-2}.
\end{align*}
This proves Lemma 2.3. 
\end{proof}
In particular we see that the error term from \eqref{partialres} is satisfactory for Theorem 2.1 as a consequence of Lemma 2.3.

It remains to show that
\begin{equation}
\label{rhodef}
    \frac{V(\vecy)}{d(\vecy)} = \varrho_\infty(\vecy).
\end{equation}

Let $\vecw = (w_1,...,w_s)$ where $w_i = y_i^2 d(\vecy)$ for $ 1 \leq i \leq s$, then one sees that $\|\vecw\|_2=1$ and
\begin{align*}
    \varrho_\infty(\vecy){d(\vecy)} = \int_{-\infty}^\infty \int_{[-1,1]^s}e(-\theta \vecw \cdot \vecx )d\vecx d\theta.
\end{align*}
In the case where only one of the components $y_i$ of $\vecy$ is nonzero, we see from remark at the beginning of this section that
\begin{align*}
    \varrho_\infty(\vecy){d(\vecy)} = \int_{-\infty}^\infty 2^{s-1} \frac{sin(2 \pi \theta)}{\pi \theta}d\theta
\end{align*}
converges conditionally to $2^{s-1}$, which suffices for Theorem 2.1 since one easily sees that $V(\vecy) = 2^{s-1}$ in this case.
In the case where $\vecy$ has more than 1 nonzero components we can rewrite the integral, again by remark at the beginning of this section, as
\begin{align*}
    \lim_{\delta \downarrow 0} \int_{-\infty}^\infty \Big( \frac{sin(\pi \delta \theta)}{\pi \delta \theta} \Big)^2 \int_{[-1,1]^s}e(-\theta \vecw \cdot \vecx )d\vecx d\theta\\
    = \lim_{\delta \downarrow 0} \int_{-\infty}^\infty  \int_{[-1,1]^s}\Big( \frac{sin(\pi \delta \theta)}{\pi \delta \theta} \Big)^2e(-\theta \vecw \cdot \vecx )d\vecx d\theta.
\end{align*}
Using results from Fourier analysis(c.f. \cite[Lemma 20.1]{Davenport}), one set
\begin{align*}
    K(u;\delta)  = \int_{-\infty}^\infty  \Big( \frac{sin(\pi \delta \theta)}{\pi \delta \theta} \Big)^2e(-\theta u ) d\theta,
\end{align*}
where
\begin{align*}
    K(u;\delta) = \begin{cases}
    \delta^{-2}(\delta-|u|), & \text{if } |u| \leq \delta,\\
    0,& \text{if } |u| \geq \delta.
    \end{cases}
\end{align*}
We then see that 
\begin{align*}
     \varrho_\infty(\vecy){d(\vecy)} = \lim_{\delta \downarrow 0}  \int_{[-1,1]^s}K(\vecw \cdot \vecx;\delta) d\vecx.
\end{align*}
Since $\|\vecw\|_2 = 1$, there exists $M \in O_s(\mathbb{R})$ such that $M \vecw = (1,0,...)$, then $\vecw \cdot \vecx = M\vecw \cdot M\vecx$. Hence, applying change of variables $\vecz = M\vecx$ and integrate over $Z=M[-1,1]^s$, one has
\begin{align*}
    \varrho_\infty(\vecy){d(\vecy)} = \lim_{\delta \downarrow 0} \int_Z K(z_1;\delta) d\vecz = meas\{\vecz \in Z; z_1 = 0\}=V(M \vecw) = V(\vecw) = V(\vecy).
\end{align*}
This proves the claim and Theorem 2.1 follows as a consequence. 

\section{Counting points using circle method}
In this section we seek to prove an asymptotic formula for the counting function
\begin{align*}
    M_3(\vecx;P) = \# \{\vecy \in \mathbb{Z}^s:|\vecy| \leq P, F(\vecx;\vecy) = 0\} 
\end{align*}
 for small $|\vecx|$. We will use the classical Hardy-Littlewood cirle method (c.f.  \cite{Davenport}). This contrasts the new $\delta$-method developed in \cite{HB} and used in \cite{Browning2}, which we have been following very closely to throughout this paper. (A reason for this comes from the fact that the classical Hardy-Littlewood circle method is applicable to a quadratic equation for $s \geq 5$, whereas Heath-Brown and Browning study the case when $s = 4$ in \cite{Browning2}).\\

In correspondence with the classical Hardy-Littlewood circle method, we make similar definitions for major and minor arcs.
Let
\begin{align*}
    \mathfrak{M}_{a,q} = \left\{\alpha \in [0,1]: \left|\alpha - \frac{a}{q}\right| < (2q|\vecx|)^{-1}P^{-1-\eta} \right\}
\end{align*}
for $a \leq q$, $(a,q) =1$, $1 \leq q \leq P|\vecx|$ and $\eta$ arbitrarily small. Define major arcs
\begin{align*}
\mathfrak{M} &= \bigcup_{q \leq P|\vecx|}\bigcup_{\substack{a=1 \\ (a,q)=1}}^q \mathfrak{M}_{a,q}
\end{align*}
and minor arcs
\begin{align*}
  \mathfrak{m} &= [0,1]\backslash \mathfrak{M}.
\end{align*}
Finally notice that
\begin{align*}
    M_3(\vecx;P) = \Bigg( \int_\mathfrak{M}+\int_\mathfrak{m} \Bigg)\Bigg(T_1(\alpha)...T_s(\alpha) d\alpha\Bigg),
\end{align*}
where
\begin{align*}
    T_j(\alpha) = \sum_{|y_j| \leq P\text{, } y_j \in \mathbb{Z}} e(\alpha x_j y_j^2).
\end{align*}

Our strategy will be as follows: along the major arcs we will be using the classical Hardy-Littlewood circle method to produce the truncated singular integral and singular series. This provides an expression immediately applicable in \S 4. For the integral along the minor arcs we will prove an upper bound, using Weyl's inequality and Hua's lemma, which should suffice for small $|\vecx|$. The case when $|\vecx|$ is large will be dealt with in \S 4. \\
\indent Our result is the following theorem. 
\begin{thm}
For $s \geq 7$ we have
\begin{equation}
    \begin{aligned}
          M_3(\vecx;P) = P^{s-2}\sigma_\infty(\vecx) \mathfrak{S}(\vecx;P)+O(T) + O(E_2)+O(E_3),
    \end{aligned}
\end{equation}
where 
\begin{align*}
    &T= \int_\mathfrak{m} T_1(\alpha)...T_s(\alpha)d\alpha \ll  E_1(\vecx;P),\\
    &\sigma_\infty(\vecx) = \int_{-\infty}^\infty \int_{[-1,1]^s} e(\theta F(\vecx;\vecy))d\vecy d\theta,\\
    &\mathfrak{S}(\vecx;P) = \sum_{q \leq P|\vecx|} q^{-s} S_q(\vecx)\text{, }S_q(\vecx) = \sum_{\substack{a \text{ mod } q\\(a,q) = 1}}\sum_{\vecb \text{ mod } q}e_q(aF(\vecx;\vecb))
\end{align*}
and 
\begin{align*}
    &E_1(\vecx;P) = P^{s/2+\eps} |\Delta(\vecx)|^{(s-4)/2s}, \\ 
    &E_2(\vecx;P) = P^{s-2-\eta}|\vecx|^{-1}{\sum_{q \leq P|\vecx|}} q^{-s/2+1} \prod_{j \leq s}(q,x_j)^{1/2}, \\
    &E_3(\vecx;P) = P^{s/2-1+c\eta} |\Delta(\vecx)|^{-1/2} |\vecx|^{s/2-1}{\sum_{q \leq P|\vecx|}} \prod_{j \leq s}(q,x_j)^{1/2}
\end{align*}
for small enough constant $c$ depending only on s, where we denote 
\begin{align*}
    \Delta(\vecx) = \prod_{j \leq s}x_j.
\end{align*}
\end{thm}
We begin by evaluating the integral along the minor arcs. 
\begin{lem}(Uniform Weyl Estimate)
Suppose that there exists natural numbers $a,q$ such that 
\begin{align*}
    (a,q)=1, q>0,\Big|\alpha-\frac{a}{q}\Big|\leq 
\frac{1}{q^2}.
\end{align*}
Then for any $\eps >0$, we have 
\begin{equation}
\label{weyl}
       \bigg|\sum_{y_j=-P}^Pe(\alpha x_j y_j^2)\bigg| \ll P^{1/2}+P^{1+\eps}\Bigg(\bigg(\frac{P}{|x_j|}\bigg)^{-{1/2}}+\bigg(\frac{q}{|x_j|}\bigg)^{-{1/2}}+\bigg(\frac{P^2}{q}\bigg)^{-1/2}\Bigg).
\end{equation}
\end{lem}

\begin{proof}
This follows straight from the classical Weyl estimate(c.f. \cite[Lemma 3.1]{Davenport}). The only difference is the replacement of the number of blocks $\frac{P}{q}+1$ with $\frac{P|x_j|}{q}+1$.
\end{proof}

\begin{lem}
We have 
\begin{align*}
    \int_0^1 |T_j(\alpha)|^{4} d\alpha \ll P^{2+\eps}.
\end{align*}
where $T_j(\alpha)$ is as previously defined. 
\end{lem}
\begin{proof}
    This is an immediate consequence of Hua's lemma. (c.f. \cite[Lemma 3.2]{Davenport} and \cite[Page 40]{Davenport}).
\end{proof}
One thus comes to an estimate for the integral over the minor arc. 
\begin{lem}
For $s \geq 7 $ we have
\begin{align*}
    T = \left| \int_\mathfrak{m} T_1(\alpha)T_2(\alpha)...T_s(\alpha)d\alpha \right| \ll E_1(\vecx;P),
\end{align*}
where $T_j(\alpha)$ is defined as before.  
\end{lem}
\begin{proof}
By H\" older's inequality
\begin{align*}
    \left| \int_\mathfrak{m} T_1(\alpha)T_2(\alpha)...T_s(\alpha) d\alpha \right| \ll \left(\int_\mathfrak{m} |T_1(\alpha)|^s d\alpha\right)^{1/s}...\left(\int_\mathfrak{m} |T_s(\alpha)|^s d\alpha\right)^{1/s},
\end{align*}
so we simply need to evaluate one of the integrals on the right-hand side. By Dirichlet approximation theorem, for every $\alpha \in [0,1]$ one can find a pair $(a,q) = 1$ with $a\leq q$ satisfying, for arbitrarily small $\eta$, that
\begin{equation}
\label{condition}
    1 \leq q \leq 2P^{1+\eta}|\vecx| \textit{, } \left|\alpha - \frac{a}{q}\right| \leq (2q|\vecx|)^{-1}P^{-1-\eta}.
\end{equation}
One sees immediately that the latter bound is exactly the one used for the major arcs. In particular if $\alpha \in \mathfrak{m}$, we must have $q \geq P|\vecx|$. Note that by our construction we have $P/|\vecx|,q/|\vecx|,P^2/q \gg P^{1-\eta}/|\vecx|$. Hence, after possibly refixing $\eta$ and $\eps$, Lemma 3.3 implies
\begin{align*}
    |T_j(\alpha)| \ll P^{s/2+\varepsilon}|x_j|^{1/2}.
\end{align*}
From Lemma 3.5 we have
\begin{align*}
    \bigg(\int_\mathfrak{m} |T_j(\alpha)|^s d\alpha\bigg)^{1/s} &\leq \max |T_j(\alpha)|^{(s-4)/s} \bigg(\int_\mathfrak{m} |T_1(\alpha)|^4 d\alpha\bigg)^{1/s}\\
    &\ll \Big(p^{(s-4)/2+\varepsilon} |x_j|^{1/2}\Big)^{1/s}.
\end{align*}
Combining using H\"older's inequality, we obtain the desired result, namely
\begin{align*}
    \int_\mathfrak{m} |T_1(\alpha)T_2(\alpha)...T_s(\alpha)|d\alpha \ll P^{s/2+\varepsilon} |\Delta(\vecx)|^{(s-4)/2s}.
\end{align*}
\end{proof}

This provides an estimate for the integral along the minor arc for small $|\vecx|$. As indicated in the beginning of this section, we will treat the integral along the minor arcs more carefully when dealing with large $|\vecx|$; see Lemma 4.2 and its proof at the end of \S 4. \\
\indent We are now ready to evaluate the integral of the exponential sums along the major arcs. This will follow from the next few lemmas. In particular, we recall a classical estimate of quadratic Gauss sums. 
\begin{lem}
For any natural numbers $a,q$ such that $(a,q) = 1$, define
\begin{align*}
    S_{x_ja,q} = \sum_{b_j \text{ mod } q} e_q(ax_jb_j^2).
\end{align*}
Then
\begin{align*}
    |S_{x_ja,q}| = q^{1/2} (q,|x_j|)^{1/2}. 
\end{align*} 
\end{lem}
\begin{proof}
    c.f. \cite[Lemma 4.8]{Iwaniek}. One can also deduce this from Lemma 3.2, setting $P = q$ and $\alpha = a/q$.
\end{proof}
\begin{lem}
For $\alpha$ in $\mathfrak{M}_{a,q}$, let $\alpha = \beta + a/q$, then we have
\begin{align*}
    T_j(\alpha) &= q^{-1}S_{x_ja,q}I_j(\beta)+O(q^{1/2}(q,x_j)^{1/2}),
\end{align*}
where
\begin{align*}
   I_j(\beta) = \int_{-P}^P e(x_j\beta u^2) du.
\end{align*}
\end{lem}
\begin{proof}
Note that one has 
\begin{align*}
    T_j(\alpha) &= \sum_{y=-P}^P e(\alpha x_j y^2)\\
    &=\sum_{r \text{ mod } q}e_q(a x_j r^2)\sum_{b}e(x_j\beta(bq+r)^2),
\end{align*}
where the second sum ranges over $b$ as $bq+r$ runs over $\mathbb{N} \cap [-P,P]$. In particular, we seek to replace the second sum by an integral as indicated, and then reevaluate the error terms. This follows immediately, or after taking complex conjugates, from van der Corput(c.f. \cite[Lemma 9.1]{Davenport}); here by our setup we have $0 < |x_j|\beta q (bq+r) \leq 1/2$ and note that the first sum is bounded by $q^{1/2}(x_j,q)^{1/2}$ from Lemma 3.6.
\end{proof}

Lemma 3.6 is crucial in providing estimate for the integral over the major arc.
\begin{lem}
For $s \geq 7$, we have
\begin{equation}
    \label{exp2}
      \int_\mathfrak{M} T_1(\alpha)...T_s(\alpha)\; d\alpha = P^{s-2}\sigma_\infty(\vecx) \mathfrak{S}(\vecx;P)+O(E_2) + O(E_3),
\end{equation}
where
\begin{align*}
       \sigma_\infty(\vecx) &= \int_{-\infty}^\infty \int_{[-1,1]^s} e(\theta F(\vecx;\vecy)) d\vecy d\theta, \\ 
       \mathfrak{S}(\vecx;P) &= \sum_{q \leq P|\vecx|} q^{-s} S_q(\vecx) \text{, } S_q(\vecx) = \sum_{\substack{a \text{ mod } q \\ (a,q)=1}}\sum_{\vecb \text{ mod } q} e_q(aF(\vecx;\vecb))
\end{align*}
and where we remind the readers from Theorem 3.1 that
\begin{align*}
    E_2 &= P^{s-2-\eta}|\vecx|^{-1}{\sum_{q \leq P|\vecx|}} q^{-s/2+1} \prod_{j \leq s}(q,x_j)^{1/2}, \\
    E_3 &= P^{s/2-1+c\eta} |\Delta(\vecx)|^{-1/2} |\vecx|^{s/2-1}{\sum_{q \leq P|\vecx|}}\prod_{j \leq s}(q,x_j)^{1/2}
\end{align*}
for some constant $c$ not depending on $\eta$.
\end{lem}
\begin{proof}
Using Lemma 3.5, we see that since
\begin{align*}
    |q^{-1}S_{x_ja,q}I_j(\beta)| \ll P q^{-1/2}(q,x_j)^{1/2},
\end{align*}
the error terms from multiplying $T_i(\alpha)$'s are bounded by $O(P^{s-1}T(q,\vecx)^{1/2}q^{-(s-2)/2})$, where
\begin{align*}
    T(q,\vecx) = \prod_{j \leq s}(q,x_j).
\end{align*}
Integrating over major arcs and summing over respective ranges for $q$ and $a$, we obtain
\begin{align*}
     \int_\mathfrak{M} T_1(\alpha)...T_s(\alpha) d\alpha &= \sum_{q \leq P|\vecx|}\sum_{\substack{a \text{ mod } q \\ (a,q)=1}}q^{-s}S_{x_1a,q}...S_{x_sa,q}\int_{\mathfrak{M}_{a,q}}\prod_{i-1}^sI_i(\alpha-a/q)\\&+ O(P^{s-2-\eta}|\vecx|^{-1}{\sum_{q \leq P|\vecx|}} \sum_{\substack{a \text{ mod } q \\ (a,q)=1}} q^{-s/2} \prod_{j \leq s}(q,x_j)^{1/2}).
\end{align*}
Applying change of variables with $\theta = P^2\beta$ and $\vecy = P\vecz$, we have
\begin{align*}
    \int_\mathfrak{M} T_1(\alpha)...T_s(\alpha) d\alpha = P^{s-2}\sum_{q \leq P}\sum_{\substack{a \text{ mod } q \\ (a,q)=1}}q^{-s}S&_{x_1a,q}...S_{x_sa,q} \int_{|\theta|  \leq P^{1-\eta}(2q|\vecx|)^{-1}}\int_{[-1,1]^s} e(\theta F(\vecx;\vecy))d\vecy d\theta \;\\ +&O(E_2).
\end{align*}
To complete the singular integral, one notices that by H\"older's inequality we have 
\begin{align*}
    \int_{|\theta| > P^{1-\eta}/(2q|\vecx|)^{-1}} \int_{[-1,1]^s} e(\theta F(\vecx;\vecy))d\vecy d\theta &\ll  \int_{|\theta| > P^{1-\eta}/(2q|\vecx|)^{-1}} \prod_{j \leq s}\min(1,|\theta|^{-1/2}|x_j|^{-1/2}) d\theta \\ 
    &\ll P^{-s/2+1+c\eta}(q|\vecx|)^{s/2-1} |\Delta(\vecx)|^{-1/2}.
\end{align*}
Summing the expression, noting that
\begin{equation}
\label{singularseriesupbound}
    |q^{-s}S_{ax_1,q}... S_{ax_s,q}| \ll q^{-s/2} T(q,\vecx)^{1/2},
\end{equation}
this leads to the error term $E_3$, after replacing the finite integral with singular integral $\sigma_\infty(\vecx)$. 
\end{proof}
We can then combine Lemma 3.8 and Lemma 3.5; Theorem 3.1 follows as a consequence. 

\section{Combination of results}
\label{section4}
This section treats the results in $\S 2,3$. In particular, we seek to sum over respective ranges for $\vecx $ and $\vecy$. To begin we remind the readers that
\begin{align*}
    M_1(\vecy; R) = \#\{\vecx \in \mathbb{Z}^s: |\vecx| \leq R, \text{ }F(\vecx ;\vecy) = 0 \} 
\end{align*}
and 
\begin{align*}
     N_1(Y; B) = \# \Bigg\{ (\vecx,\vecy) &\in \mathbb{Z}^s \times \zprim: \begin{aligned}
        F&(\vecx;\vecy) = 0,\\|\vecx|^{s-1}|\vecy|^{s-2} &\leq B, Y \leq |\vecy| <  2Y
    \end{aligned} \Bigg\}.
\end{align*}
By construction, $N_1(Y;B)$ can be rewritten as
\begin{align*}
        N_1(Y; B) = \sum_{\substack{\vecy \in \zprim\\ Y \leq |\vecy| < 2Y}}M_1(\vecy;{(B/|y|^{s-2})}^{1/(s-1)}).
\end{align*}
We also define
\begin{align*}
    M_2(\vecy;R) = \#\{\vecx \in \zprim: |\vecx| \leq R, \text{ }F(\vecx ;\vecy) = 0 \} 
\end{align*}
and let 
\begin{align*}
     N_2(Y; B) = \# \Bigg\{ (\vecx,\vecy) &\in \zprim \times \zprim: \begin{aligned}
        F&(\vecx;\vecy) = 0,\\|\vecx|^{s-1}|\vecy|^{s-2} &\leq B, Y \leq |\vecy| <  2Y
    \end{aligned} \Bigg\} .
\end{align*}
By construction we see that
\begin{align*}
    N_2(Y;B) = \sum_{\substack{\vecy \in \zprim\\ Y \leq |\vecy| < 2Y}}M_2(\vecy;{(B/|y|^{s-2})}^{1/(s-1)}).
\end{align*}
Notice that $N_2(Y;B)$ is the expression we seek to evaluate in Theorem 1.1; the following lemma gives an asymptotic for this expression. 
\begin{lem}
For $1 \leq Y \leq B^{1/(s+2)}$ we have 
\begin{align*}
    N_2(B;Y) &= \frac{B}{\zeta(s-1)} \sum_{\substack{\vecy \in \zprim\\ Y \leq |\vecy| < 2Y}}\frac{\varrho_\infty(\vecy)}{|y|^{s-2}}+O(B^{1/(s-1)}Y^{s-(s-2)/(s-1)})\\&+O(B^{(s-2)/(s-1)}Y^{(s+2)/(s-1)})+O(B^{1/(s-1)}Y^{{(s-2)}^2/(s-1)}).
\end{align*}
\end{lem}

\begin{proof}
Using \eqref{partialres}, we note that the proof is completely analogous to \cite[Lemma 5.1]{Browning2}. 
\end{proof}
One notices that when $Y \leq B^{1/(s+2)}$, all 3 error terms contribute $O(B)$. In particular one sees that after dyadic summation this leads to
\begin{align*}
      \# \Bigg\{ (\vecx,\vecy) &\in \zprim \times \zprim: \begin{aligned}
        F&(\vecx;\vecy) = 0;\\|\vecx|^{s-1}|\vecy|^{s-2} &\leq B;|\vecy| \leq B^{1/{(s+2)}}
    \end{aligned} \Bigg\} \\&=  \frac{B}{\zeta(s-1)}\sum_{\substack{\vecy \in \zprim\\ |y| \leq B^{1/(s+2)}}}\frac{\varrho_\infty(\vecy)}{|y|^{s-2}}+O(B).
\end{align*}
Next we turn to the circle method treatment in \S 3; we remind the readers that
\begin{align*}
    M_3(\vecx;P) &= \# \{\vecy \in \mathbb{Z}^s:|\vecy| \leq P, F(\vecx;\vecy) = 0\},
\end{align*}
and we define
\begin{align*}
     N_3(X;B) = \# \Bigg\{ (\vecx,\vecy) &\in \zprim \times \mathbb{Z}^s: \begin{aligned}
        F&(\vecx;\vecy) = 0,\\|\vecx|^{s-1}|\vecy|^{s-2} &\leq B,X \leq |\vecx| < 2X
    \end{aligned} \Bigg\}.
\end{align*}
Then it follows from the construction that
\begin{align*}
    N_3(X;B) &= \sum_{\substack{X \leq |\vecx| < 2X \\ \vecx \in \zprim}} M_3(\vecx;(B/|\vecx|^{s-1})^{1/(s-2)}).
\end{align*}
We make use of the following lemma about minor arc contribution, which shall be proved at the end of this section. 
\begin{lem}
For  $1 \leq X \leq B^{4/(s+2)(s-1)}$ and some small enough $c_5$, we have
\begin{align*}
    \sum_{\substack{X \leq |\vecx| < 2X \\ \vecx \in \zprim}} T \ll B^{1-c_5},
\end{align*}
where $T$ is as defined in Theorem 3.1. 
\end{lem}
Assuming Lemma 4.2, we provide an estimate for $N_3(X;B)$.
\begin{lem}
For  $1 \leq X \leq B^{4/(s+2)(s-1)}$ and some small enough $c_4$ we have
\begin{align*}
     N_3(X;B) =  B \sum_{\substack{X \leq |\vecx| < 2X \\ \vecx \in \zprim}} \frac{\sigma_\infty(\vecx) \mathfrak{S}(\vecx)}{|\vecx|^{s-1}} + O(B^{1-c_4}),
\end{align*}
 where 
\begin{align*}
    \mathfrak{S}(\vecx) = \sum_{q = 1}^\infty q^{-s} S_q(\vecx) \text{, } S_q(\vecx) = \sum_{\substack{a \text{ mod } q \\ (a,q)=1}}\sum_{\vecb \text{ mod } q} e_q(aF(\vecx;\vecb)).
\end{align*}
\end{lem}
\begin{proof}
We wish to sum $M_3(\vecx;P)$ over the range $|\vecx| \sim X$ and $\vecx \in \zprim$. It suffices to evaluate each error term in Theorem 3.1, namely $T$, $E_2$ and $E_3$, and to replace the finite sum by a singular series with another sufficiently bounded error term.\\ 
We begin by evaluating the error terms. Following from Lemma 4.2, it suffices to evaluate $E_2$ and $E_3$. Recall that
\begin{align*}
    E_2(\vecx;P) &= P^{s-2-\eta}|\vecx|^{-1}{\sum_{q \leq P|\vecx|}}  q^{-s/2+1} \prod_{j \leq s}(q,x_j)^{1/2}, \\
    E_3(\vecx;P) &= P^{s/2-1+c\eta} |\Delta(\vecx)|^{-1/2} |\vecx|^{s/2-1}{\sum_{q \leq P|\vecx|}}\prod_{j \leq s}(q,x_j)^{1/2}.
\end{align*}

We first seek to evaluate $E_2$. Note that
\begin{align*}
    \sum_{\substack{X \leq |\vecx| < 2X \\ \vecx \in \zprim}} E_2(\vecx;(B/|\vecx|^{s-1})^{1/(s-2)}) &\ll \frac{B^{1-c\eta}}{X^{s-1-d\eta}}X^{-1}\sum_\vecx \sum_{q \leq (B/|\vecx|^{s-1})^{1/(s-2)}|\vecx|} q^{-s/2+1}\prod_{j \leq s}(q,x_j)^{1/2} \\
    &\ll \frac{B^{1-c\eta}}{X^{s-d\eta}}\sum_{q \leq (B/|\vecx|^{s-1})^{1/(s-2)}|\vecx|} q^{-s/2+1} \prod_{j \leq s} \sum_{\substack{x_j \in \mathbb{Z} \\ |x_j| \leq X}} (q,x_j)^{1/2},
\end{align*}
where $c,d$ are some small fixed positive constants depending only on $s$.
For each individual sum in the final product we see that 
\begin{align*}
    \sum_{\substack{x_j \in \mathbb{Z} \\ x_j \leq X}} (q,x_j)^{1/2} &\ll \sum_{\substack{u_1 \leq X \\ u_1 \mid q^\infty}} (q,u_1)^{1/2}\sum_{\substack{u_2 \leq X/u_1 \\ (u_2,q) = 1}} 1\\
    &\ll X \sum_{\substack{u_1 \leq X \\ u_1 \mid q^\infty}} \frac{(q,u_1)^{1/2}}{u_1}\\
    &\ll_\eps X^{1+\eps}q^\eps,
\end{align*}
where the last inequality follows from that there are $O(X^\eps q^\eps)$ choices in the last sum, which one may obtain by applying Rankin's method to the sum. 
In particular we have 
\begin{align*}
     \sum_{\substack{X \leq |\vecx| < 2X \\ \vecx \in \zprim}} E_2(\vecx;(B/|\vecx|^{s-1})^{1/(s-2)}) &\ll_\eps X^{s+\eps} \frac{B^{1-c\eta}}{X^{s-d\eta}} \sum_{q }q^{-s/2+1+\eps}, \\ 
     &\ll_{\eps,\eta} B^{1-c_2},
\end{align*}
for small enough $c_2$ and $X \leq B^{4/(s+2)(s-1)}$. 

To evaluate the contribution of $E_3$ we start by assuming that $|x_j| \sim X$ for all $j \leq s$. It's easy to see, following similar process as $E_2$, that
\begin{align*}
    \sum_{\substack{X \leq |\vecx| < 2X \\ \vecx \in \zprim}} E_3(\vecx;(B/|\vecx|^{s-1})^{1/(s-2)}) &\ll_{\eta,\eps} X^{s+\eps}(B/X^{s-1})^{(s/2+c\eta)/(s-2)} \ll B^{1-c_3}
\end{align*}
for small enough $c_3,\eta$ and $X \leq B^{4/(s+2)(s-1)}$. We can then remove the condition $|x_j| \sim X$ by the following argument: if we assume that $X^{1-p_j} \gg |x_j| \gg X^{1-t_j}$ for some $t_j$ then we would be replacing $X^s$ with $X^{s-\sum p_j}$ and $|\Delta(\vecx)|^{-1/2}|\vecx|^{s/2-1}$, originally treated as $X^{-1}$, with $X^{\sum t_j/2-1}$. Choosing appropriate and finite number of $p_j$'s and $t_j$'s, one can show that the contribution from this range won't exceed $B^{1-c_3/2}$, which is satisfactory. 

At this point we note that after summing
\begin{align*}
       N_3(X;B) =  B \sum_{\substack{X \leq |\vecx| < 2X \\ \vecx \in \zprim }} \frac{\sigma_\infty(\vecx) \mathfrak{S}(\vecx;B^{1/(s-2)}|\vecx|^{1/(s-2)})}{|\vecx|^{s-1}} + O(B^{1-c_4}),
\end{align*}
so it suffices to establish the validity of replacing the finite sum with the singular series.
We begin by noting that
\begin{align*}
    \sum_{q \leq C}q^{-s/2-2}S_q(\vecx)\ll \sum_{q \leq C}\frac{\prod_{j \leq s}(x_j,q)^{1/2}}{q} \ll C |\Delta_\text{bad}(\vecx)|^{(s-1)/2s},
\end{align*}
where
\begin{align*}
    \Delta_\text{bad}{(\vecx)} = \prod_{\substack{p \mid \Delta(\vecx) \\ \nu_p(\Delta(\vecx)) \geq 2}} p^{\nu_p(\Delta(\vecx))}.
\end{align*}
Hence by partial summation
\begin{align*}
     \sum_{q \geq (B|\vecx|)^{1/(s-2)}}q^{-s}S_q(\vecx)\ll (B|\vecx|)^{(-s/2+3)/(s-2)} |\Delta_\text{bad}(\vecx)|^{(s-1)/2s}.
\end{align*}
Since $\sigma_\infty(\vecx) \ll |\Delta(\vecx)|^{-1/2}$ (c.f. \cite[Theorem 7.1]{Davenport}) and $|\vecx|^{s-1} \geq  |\Delta(\vecx)|^{(s-1)/s}$ we see that the tail of the singular series contributes
\begin{align*}
    \ll (BX)^{(-s/2+3)/(s-2)} \sum_{\substack{X \leq |\vecx| < 2X \\ \vecx \in \zprim}} \frac{|\Delta_\text{bad}(\vecx)|^{(s-1)/2s}}{|\Delta(\vecx)|^{1/2+(s-1)/s}}.
\end{align*}
Rewriting $n = |\Delta_\text{bad}(\vecx)|$ and $t = |\Delta(\vecx)|$, one has
\begin{align*}
    &\ll (BX)^{(-s/2+3)/(s-2)} \sum_{\substack{n \ll X^{s} \\ n \text{ squarefull}}} n^{(s-1)/2s} \sum_{\substack{t \ll X^{s} \\ n \mid t}} \frac{\tau_s(t)}{t^{1/2+(s-1)/s}}\\
    &\ll B^{-c} \sum_{n \text{ squarefull}} \frac{1}{n^{(s+1)/2s}}.
\end{align*}
for small positive constant $c$ where $s \geq 7$ and $X \leq B^{4/(s+2)(s-1)}$. This concludes the proof, noting that the last sum converges.
\end{proof}

Similar as before, define
\begin{align*}
    M_4(\vecx;P) &= \# \{\vecy \in \zprim;|\vecy| \leq P; F(\vecx;\vecy) = 0\}
\end{align*}
and let 
\begin{align*}
      N_4(X;B) &= \# \Bigg\{ (\vecx,\vecy) &\in \zprim \times \zprim: \begin{aligned}
        F&(\vecx;\vecy) = 0,\\|\vecx|^{s-1}|\vecy|^{s-2} &\leq B,X \leq |\vecx| < 2X
    \end{aligned} \Bigg\}.
\end{align*}
By construction we have
\begin{align*}
    N_4(X;B) &= \sum_{\substack{X \leq |\vecx| < 2X \\ \vecx \in \zprim}} M_4(\vecx;(B/|\vecx|^{s-1})^{1/(s-2)}).
\end{align*}
Notice that the expression $N_4(X;B)$ is exactly what we are looking for in Theorem 1.1 and thus we seek an asymptotic for it. 
\begin{lem}
For $1 \leq X \leq B^{4/(s+2)(s-1)}$, we have 
\begin{align*}
     N_4(X;B) = \frac{B}{\zeta(s-2)}& \sum_{\substack{X \leq |\vecx| < 2X \\ \vecx \in \zprim}} \frac{\sigma_\infty(\vecx) \mathfrak{S}(\vecx)}{|\vecx|^{s-1}}+O(B^{1-\eta})
\end{align*}
for  small enough $\eta$.
\end{lem}
\begin{proof}
By previous lemma we have for suitably small $\eta$
\begin{align*}
    N_4(X;B) &= \sum_{d \leq ({B/X^{s-1}})^{1/(s-2)}}\mu(d)N_3(X;B/d^{s-2})\\
    &=B\sum_{d \leq ({B/X^{s-1}})^{1/(s-2)}} \frac{\mu(d)}{d^{s-1}} \sum_{\substack{X \leq |\vecx| < 2X \\ \vecx \in \zprim}}\frac{\sigma_\infty(\vecx) \mathfrak{S}(\vecx)}{|\vecx|^{s-1}}+O(B^{1-\eta}).
\end{align*}
Replacing the summation over m\"obius functions $\mu$ with a series, we conclude that 
\begin{align*}
     N_4(X;B) = \frac{B}{\zeta(s-2)}& \sum_{\substack{X \leq |\vecx| < 2X \\ \vecx \in \zprim}} \frac{\sigma_\infty(\vecx) \mathfrak{S}(\vecx)}{|\vecx|^{s-1}}+O(B^{1-\eta}).
\end{align*}
\end{proof}
From Lemma 4.1 and 4.4 we may establish that
\begin{align*}
    \# \Bigg\{ (\vecx,\vecy) &\in \zprim \times \zprim: \begin{aligned}
        F&(\vecx;\vecy) = 0;\\|\vecx|^{s-1}|\vecy|^{s-2} &\leq B;|\vecx| \leq B^{4/{(s+2)(s-1)}}
    \end{aligned} \Bigg\} \\&=  \frac{B}{\zeta(s-2)}  \sum_{\substack{|\vecx| \in \zprim \\ \vecx| \leq B^{4/(s+2)(s-1)}}} \frac{\sigma_\infty(\vecx) \mathfrak{S}(\vecx)}{|\vecx|^{s-1}}+O(B^{1-\eta}\log B)+ O(B),
\end{align*}
and in particular
\begin{align*}
    N(\Omega,B) &= \frac{1}{4}\Big(\frac{B}{\zeta(s-2)} \sum_{\substack{|\vecx| \in \zprim \\ |\vecx| \leq B^{4/(s+2)(s-1)}}} \frac{\sigma_\infty(\vecx) \mathfrak{S}(\vecx)}{|\vecx|^{s-1}}+\frac{B}{\zeta(s-1)}\sum_{\substack{\vecy \in \zprim\\ |\vecy| \leq B^{1/(s+2)}}}\frac{\varrho_\infty(\vecy)}{|\vecy|^{s-2}}\Big)\\&+O(B)\\
    &=\frac{B}{4\zeta(s-2)}M_1(B) + \frac{B}{4\zeta(s-1)}M_2(B)+O(B),
\end{align*}
with 
\begin{align*}
    M_1(B) = \sum_{\substack{|\vecx| \in \zprim \\ |\vecx| \leq B^{4/(s+2)(s-1)}}} \frac{\sigma_\infty(\vecx) \mathfrak{S}(\vecx)}{|\vecx|^{s-1}} \text{ and } M_2(B) = \sum_{\substack{\vecy \in \zprim\\ |\vecy| \leq B^{1/(s+2)}}}\frac{\varrho_\infty(\vecy)}{|\vecy|^{s-2}}.
\end{align*}

It remains to prove Lemma 4.2. 
\begin{proof}
We recall that it suffices to show
\begin{align*}
      \sum_{\substack{X \leq |\vecx| < 2X \\ \vecx \in \zprim}} T \ll B^{1-c_5},
\end{align*}
for some small enough $c_5$. We remind the readers that 
\begin{align*}
    T &= \int_{\mathfrak{m}(\vecx)} T_1(\alpha)...T_s(\alpha)d\alpha 
\end{align*}
where $\mathfrak{m}(\vecx)$ is the minor arc associated with $\vecx$. 
Our strategy will be as follows: when $X$ is small, we will apply Lemma 3.4 to provide a satisfying upper bound; when $X$ is large, we will treat the sum over $|\vecx| \leq X$ and $|\vecy| \leq P$ as a cubic form in $2s$ variables, with $P = (B/X^{s-1})^{1/(s-2)}$.
To begin, we recall from Lemma 3.4 that 
\begin{align*}
    \sum_{\substack{X \leq |\vecx| < 2X \\ \vecx \in \zprim}} T   &\ll \sum_{\substack{X \leq |\vecx| < 2X \\ \vecx \in \zprim}} E_1(\vecx;(B/X^{s-1})^{1/(s-2)}) \\
    &\ll X^s \frac{B^{s/2+\eps}}{X^{s(s-1)/2(s-2)}}X^{(s-4)/2} \ll B^{1-c_5} 
\end{align*}
for $s \geq 8$, $X \leq B^{4/(s+2)(s-1)} $ or $s = 7$, $X \leq B^{3.46/(s+2)(s-1)}$ and $c_5 >0$ small enough. 
We now make the definition
\begin{align*}
    C(\vecu) = \sum_{i=1}^s u_i u_{2s+1-i}^2
\end{align*} 
for $C$ cubic form on $\mathbb{Z}^{2s}$. In particular, one notices that 
\begin{equation}
\label{ineq}
      \sum_{\substack{X \leq |\vecx| < 2X \\ \vecx \in \zprim}} T \ll \sup_{\alpha} \left|\sum_{\substack{|\vecx| \sim X, \Delta(\vecx) \neq 0 \\ |\vecy| \leq P \\ \vecx,\vecy \in \mathbb{Z}^s}} e(\alpha C(\vecx,\vecy))\right|.
\end{equation}
where $\sup$ is taken over all $\alpha$ in the minor arcs $\mathfrak{m}(\vecx)$; in particular, if there exists $(a,q)=1$ such that $|\alpha-a/q| \leq 1/2(qXP^{1+\eta})$, then $q \geq PX$. One should think of this as taking supremum over unions of minor arcs for each $\vecx \in \mathbb{Z}^s$(c.f. Lemma 3.4), and in particular, for all such $\alpha$, we have $PX \leq q \leq 4PX$.
\\ Denote the right hand sum by $S(\alpha)$, then tracing the arguments of \cite[\S 2]{HB2}, one finds that 
\begin{align*}
    |S(\alpha)|^4 \ll P^{3s+\eps} X^s \sum_{\substack{|u_1...u_s| \neq 0\text{, }\max_{i \leq s}|u_i| \leq X \\ |u_{s+1}|...|u_{2s}| \leq P}}M(\alpha,P,X),
\end{align*}
where 
\begin{align*}
    M(\alpha,P,X) &:=\# \bigg\{ \vecv \in \mathbb{Z}^{2s}:\begin{aligned}
       & 0 < |v_1|...|v_s| \leq X, |v_{s+1}|...|v_{2s}| \leq P,\\
       & ||6\alpha B_i(\vecu,\vecv)|| \leq P^{-1}  \text{ }  \forall i \leq 2s
    \end{aligned} \bigg\} 
\end{align*}
and where 
\begin{align*}
    B_i(\vecu,\vecv) = \begin{cases}
        \ \frac{1}{3}u_{2s+1-i} v_{2s+1-i} &\text{ if } i \leq s, \\
        \ \frac{1}{3}u_i v_{2s+1-i} + \frac{1}{3}v_i u_{2s+1-i} &\text{ if } i > s. \\
    \end{cases}
\end{align*}
Then trivially we have
\begin{align*}
    M(\alpha,P,X) \ll  \# \bigg\{ \vecv \in \mathbb{Z}^{2s}:\begin{aligned}
       & 0 < |v_1|...|v_{s}|, |v_1|...|v_{2s}| \leq P\\
        &\|6\alpha B_i(\vecu,\vecv)\| \leq P^{-1}  \text{ }  \forall i \leq 2s
    \end{aligned} \bigg\}.
\end{align*}
It therefore follows from \cite[Lemma 2.2, 2.3]{HB2} and \cite[Equation 2.7]{HB2} that 
 \begin{align*}
     |S(\alpha)|^4 \ll X^{3s}P^{s} Z^{-4s}  \# \left\{ (\vecu, \vecv) \in \mathbb{Z}^{4s}:\begin{aligned}
     & |u_1...u_s|,|v_1...v_s| \neq 0\\
       & |v_1|...|v_{2s}|,|u_1|...|u_{2s}| \leq ZP\\
        &\|6\alpha B_i(\vecu,\vecv)\| = 0  \text{ }  \forall i \leq 2s
    \end{aligned} \right\},
\end{align*}
with 
\begin{align*}
    0 < Z < 1, \text{ } Z^2 < (12cq|\beta|P^2)^{-1}, Z^2 < P/2q
\end{align*}
 and
 \begin{align*}
     Z^2 < \max \Big(\frac{q}{6cP^2}, qP|\theta|\Big),
 \end{align*}
 where $\alpha = a/q+\beta$ is the rational approximation as indicated before (or c.f. Lemma 3.4 with $|\vecx|$ replaced by $X$). In particular, from our remark at the beginning of the proof of this lemma, $4P^{1+\eta}X \geq q \geq PX$.  Finally notice that there is only 1 solution to $|6\alpha B_i(\vecu,\vecv)\| = 0$ with $u_1...u_s,v_1...v_s \neq 0$, namely when $ v_{s+1} = ... = v_{2s} = u_{s+1} =...= u_{2s} = 0 $ by definition of the bilinear forms $B_i$'s. 
 
 From previous discussions and \eqref{ineq}, we obtain
 \begin{align*}
      \left(\sum_{\substack{X \leq |\vecx| < 2X \\ \vecx \in \zprim}} T\right)^{4} \ll X^{3s}P^{s+\eta}(1+(q|\theta| P^2)^{2s}+ q^{2s}P^{-2s} + P^{2s}q^{-2s}). 
 \end{align*}
 Finally by our construction of minor arcs and $\alpha$, we have that $\theta \ll 1/qXP^{1+\eta}$ and $PX \leq q \leq 2P^{1+\eta}X$. Thus for $B^{3.46/(s+2)(s-1)} \leq  X \leq B^{4/(s+2)(s-1)}$, we deduce that
 \begin{align*}
     \sum_{\substack{X \leq |\vecx| < 2X \\ \vecx \in \zprim}} T \ll X^{3s/4}P^{3s/4+\eta} + P^{5s/4+\eta} X^{-s/4} \ll B^{1-c_5},
 \end{align*}
which holds for $P = (B/X^{s-1})^{1/(s-2)}$, $\eta$ small enough and $s \geq 7$. We see that this is satisfactory for Lemma 4.2. (When $s = 6$, the error term becomes $O(B^{81/80+\eta})$, which is not satisfactory). 
\end{proof}

\section{Finale}
In this section we establish asymptotic formulae for $M_1(B)$ and $M_2(B)$ as given in \S 4; we will then show that the overall contribution is satisfactory for Theorem 1.1. Again we follow very closely with the work of Heath-Brown and Browning \cite{Browning2}. \\

We begin by evaluating $M_2(B)$ as it is slightly easier to handle. Recall that 
\begin{align*}
    M_2(B) = \sum_{\substack{\vecy \in \zprim\\ |\vecy| \leq B^{1/(s+2)}}}\frac{\varrho_\infty(\vecy)}{|\vecy|^{s-2}},
\end{align*}
where 
\begin{align*}
    \varrho_\infty(\vecy) = \int_{-\infty}^\infty \int_{[-1,1]^s}e(\theta F(\vecx;\vecy))d\vecx d\theta.
\end{align*}
\begin{lem}
For $s \geq 7$, we have
\begin{align*}
        M_2(B) = \frac{(s-2)}{(s+2)\zeta(s)}\tau_\infty \log B +O(1),
\end{align*}
where $\tau_\infty$ is as given in Theorem 1.1.
\end{lem}
To begin, we add in M\"obius functions to check for primitivity conditions. We have
\begin{align*}
    M_2(B) = \sum_{k \leq B^{1/(s+2)}}\frac{\mu(k)}{k^{s-2}} \sum_{\vecy \in \mathbb{Z}^s \cap T_0 } \frac{\varrho_\infty(k\vecy)}{|\vecy|^{s-2}},
\end{align*}
where 
\begin{align*}
    T_0 = T_0(k) = \{\vecy \in R^s; 1 \leq |\vecy| \leq B^{1/(s+2)}/k\}.
\end{align*}
Noting that $\varrho_\infty(k \vecy) =k^{-2} \varrho_\infty(\vecy)$, we have 
\begin{equation}
\label{inm2}
    M_2(B) = \sum_{k \leq B^{1/(s+2)}}\frac{\mu(k)}{k^{s}} \sum_{\vecy \in \mathbb{Z}^s \cap T_0 } \frac{\varrho_\infty(\vecy)}{|\vecy|^{s-2}}.
\end{equation}
The main idea is then to convert the second sum into an integral around $[-1,1]^s$ to produce $\tau_\infty$ from $\varrho_\infty$. Note that 
\begin{align*}
    \varrho_\infty(\vecy) \ll |\vecy|^{-2},
\end{align*}
which is clear from \eqref{rhodef}. Using this fact, one establishes accordingly from \cite[Lemma 6.2]{Browning2}, changing $4$ variables to $s$ variables when needed, that
\begin{lem}
For $\min_j(|y_j|) \geq 2$, we have 
\begin{equation}
    \label{sumtoint1}
    \frac{ \varrho_\infty(\vecy)}{|\vecy|^{s-2}} = \int_{[0,1]^s} \frac{ \varrho_\infty(\vecy+\vect)}{|\vecy+\vect|^{s-2}} d\vect +O(|\vecy|^{2-s}\min_j(|y_j|)^{-(s-3)/(s-1)} |\Delta(\vecy)|^{-2/(s-1)}).
\end{equation}
\end{lem}
It suffices to replace the second sum in \eqref{inm2} with an integral.
\begin{lem}
We have
\begin{align*}
\sum_{\vecy \in \mathbb{Z}^s \cap T_0} \frac{\varrho_\infty(\vecy)}{|\vecy|^{s-2}} = J_2(B;k)+O(1), 
\end{align*}
where
\begin{align*}
    J_2(B;k) = \int_{T_0} \frac{\varrho_\infty(\vecy)}{|\vecy|^{s-2}} d\vecy.
\end{align*}
\end{lem}
\begin{proof}
We begin by defining 
\begin{align*}
    X = \{\vecy \in \mathbb{Z}^s: |\vecy| \leq B^{1/(s+2)}/k-2, \min_j|y_j| \geq 2\},
\end{align*}
and let 
\begin{align*}
    Y = \bigcup_{\substack{\vecy \in X}}\vecy + (0,1]^s.
\end{align*}
Note that $Y \subset T_0$ amd $T_0 \setminus Y \subset T_1 \cup T_2$, where
\begin{align*}
    T_1 = \{\vect \in T_0: B^{1/(s+2)}/k -3 \leq |\vecy| \leq B^{1/(s+2)}/k\},
\end{align*}
and 
\begin{align*}
     T_2 = \{\vect \in T_0: \min_j|y_j| \leq 3\}.
\end{align*}
Thus using the previous lemma we see that 
\begin{align*}
    \sum_{\vecy \in \mathbb{Z}^s \cap T_0} \frac{\varrho_\infty(\vecy)}{|\vecy|^{s-2}} &= \sum_{\vecy \in X}\frac{\varrho_\infty(\vecy)}{|\vecy|^{s-2}} + O(\sum_{\vecy \in T_1}+\sum_{\vecy \in T_2} \frac{\varrho_\infty(\vecy)}{|\vecy|^{s-2}})\\
    &=\int_{T_0} \frac{\varrho_\infty(\vecy)}{|\vecy|^{s-2}} d\vecy + O\Big(\sum_{i=0}^2 E_i\Big),
\end{align*}
where
\begin{align*}
    E_0 = \sum_{\substack{\vecy \in \mathbb{Z}^s \cap T_0 \\ \min_j|y_j| \geq 2}} |\vecy|^{2-s}\min_j(|y_j|)^{-(s-3)/(s-1)} |\Delta(\vecy)|^{-2/(s-1)}
\end{align*}
and 
\begin{align*}
    E_i \ll \sum_{\vecy \in \mathbb{Z}^s \cap T_i} |\vecy|^{-s} + \int_{T_i} |\vecy|^{-s} d\vecy
\end{align*}
for $i = 1,2$. In particular, one easily see that $E_i \ll 1$ for $i = 1,2$, and a similar claim holds for $E_0$, by using the same argument as in the proof of Lemma 4.3 in assuming that $\vecy \sim |\vecy|$. This concludes the proof. 
\end{proof}
To prove Lemma 5.1 it remains to investigate $J_2(B;k)$. 
\begin{lem}
Let $k \leq B^{1/(s+2)}$. Then
\begin{align*}
    J_2(B;k) = (s-2) \tau_\infty \log (B^{1/(s+2)}/k).
\end{align*}
\end{lem}
\begin{proof}
We note that $\varrho_\infty(\vecy)$ is unchanged in response to permuting coordinates, and that $\varrho_\infty(\vecy) = |\vecy|^{-2}\varrho_\infty(t_1...t_{s-1},t_s)$, when $|y_s| = |\vecy|$ and $t_i = y_i/|\vecy|$ for $i \leq s-1$. The rest of the proof is analogous to \cite[Lemma 6.4]{Browning2}. 
\end{proof}
To prove Lemma 5.1 we note that 
\begin{align*}
    M_2(B) &= \sum_{k \leq B^{1/(s+2)}}\frac{\mu(k)}{k^{s}} \sum_{\vecy \in \mathbb{Z}^s \cap T_0 } \frac{\varrho_\infty(\vecy)}{|\vecy|^{s-2}}\\
    &= \sum_{k \leq B^{1/(s+2)}}\frac{\mu(k)}{k^{s}} J_2(B;k) +O(1),
\end{align*}
which follows from Lemma 5.3. 
From
\begin{align*}
    \sum_{k \leq B^{1/(s+2)}}\frac{\mu(k)}{k^{s}} = \zeta(s)^{-1} +O(B^{-(s-1)/(s+2)}),
\end{align*}
Lemma 5.4 implies that 
\begin{align*}
    M_2(B) = \frac{(s-2)}{(s+2)\zeta(s)} \log B + O(1).
\end{align*}
Concluding the proof of lemma 5.1.

We then turn our attention to $M_1(B)$. Again we remind the readers that 
\begin{align*}
    M_1(B) = \sum_{\substack{\vecx \in \zprim \\ |\vecx| \leq B^{4/(s+2)(s-1)}}} \frac{\sigma_\infty(\vecx) \mathfrak{S}(\vecx)}{|\vecx|^{s-1}},
\end{align*}
where 
\begin{align*}
    \sigma_\infty(\vecx) &= \int_{-\infty}^\infty \int_{[-1,1]^s} e(\theta F(\vecx;\vecy))d\vecy d\theta,\\
    \mathfrak{S}(\vecx) &= \sum_{q=1}^\infty q^{-s} S_q(\vecx)\text{, }S_q(\vecx) = \sum_{\substack{a \text{ mod } q\\(a,q) = 1}}\sum_{\vecb \text{ mod } q}e_q(aF(\vecx;\vecb)).
\end{align*}
The remainder of this section is devoted to producing an estimate for $M_1(B)$. 
\begin{lem}
Let $s \geq 7$. Then
\begin{align*}
    M_1(B) = \frac{4\zeta(s-2)}{(s+2)\zeta(s)\zeta(s-1)}\tau_\infty \log B + O(1).
\end{align*}
\end{lem}

Recall from \S 4 that  
\begin{align*}
     M_1(B) = \sum_{\substack{|\vecx| \in \zprim \\ |\vecx| \leq B^{4/(s+2)(s-1)}}} \frac{\sigma_\infty(\vecx) \mathfrak{S}(\vecx)}{|\vecx|^{s-1}}
\end{align*}
where
\begin{align*}
    \mathfrak{S}(\vecx) = \sum_{q = 1}^\infty q^{-s} S_q(\vecx) \text{, } S_q(\vecx) = \sum_{\substack{a \text{ mod } q \\ (a,q)=1}}\sum_{\vecb \text{ mod } q} e_q(aF(\vecx;\vecb)).
\end{align*}

We will first show that the singular series can be truncated to $q \leq B^{\eta}$ for arbitrarily small $\eta > 0$; this allows for the switching of sums in $q$ and $\vecx$. We see from \cite[Theorem 7.1]{Davenport} that $\sigma_\infty(\vecx) \leq |\Delta(\vecx)|^{-1/2}$ and $|\vecx|^{s-1} \geq |\Delta(\vecx)|^{(s-1)/s}$. By applying \eqref{singularseriesupbound} and partial summation, we see that
\begin{align*}
   \sum_{\substack{\vecx \in \zprim \\ |\vecx| \leq B^{4/(s+2)(s-1)}}} \frac{\sigma_\infty(\vecx)} {|\vecx|^{s-1}} \sum_{q > B^{\eta}} q^{-s} S_q(\vecx)   &\ll  \sum_{\substack{|\vecx| \in \zprim \\ |\vecx| \leq B^{4/(s+2)(s-1)}}} \frac{\sigma_\infty(\vecx)} {|\vecx|^{s-1}} \sum_{q > B^{\eta}} q^{-s/2+1}\prod_{j \leq s} (q,x_j)^{1/2}\\
   &\ll \sum_{q > B^{\eta}} q^{-s/2+1} \prod_{j \leq s} \sum_{\substack{x_j \in \mathbb{Z} \\ |x_j| \leq B^{4/(s+2)(s-1)}}} \frac{(q,x_j)^{1/2}}{x_j^{1/2+(s-1)/s}}.
\end{align*}
Notice that the inner sum inside the product is bounded by 
\begin{align*}
    \sum_{\substack{u_1 \leq B^{4/(s+2)(s-1)} \\ u_1 \mid q^\infty}} \frac{(q,u_1)^{1/2}}{u_1^{1/2+(s-1)/s}} \sum_{\substack{u_2 \leq B^{4/(s+2)(s-1)}/u_1 \\ (u_2,q)=1}} \frac{1}{u_2^{1/2+ (s-1)/s}} \ll_\varepsilon q^\varepsilon
\end{align*}
as the inner sum is absolutely convergent and the outer sum has at most $O(B^\varepsilon q^\varepsilon)$ terms, which one may obtain by applying Rankin's method; the bound is then obtained using partial summation. In particular,
\begin{align*}
     \sum_{\substack{\vecx \in \zprim \\ |\vecx| \leq B^{4/(s+2)(s-1)}}} \frac{\sigma_\infty(\vecx)} {|\vecx|^{s-1}} \sum_{q > B^{\eta}} q^{-s} S_q(\vecx) \ll B^{-\eta(s/2+2+\varepsilon)},
\end{align*}
which is sufficient. This allows us to remove the tail term of the singular series with an error term of $O(B^{-\eta(s/2-2+\varepsilon)})$.

We will then show that we can also truncate the sums in $\vecx$ to only those satisfying $|\vecx| \geq B^{3s\eta}$. Indeed, similar as previous computation, we see that 
\begin{align*}
     \sum_{\substack{\vecx \in \mathbb{Z}^s \\ |\vecx| \leq B^{3s\eta}}} \frac{\sigma_\infty(\vecx)} {|\vecx|^{s-1}} \sum_{q \leq B^{\eta}} q^{-s} S_q(\vecx)  
   &\ll \sum_{q \leq B^{\eta}} q^{-s/2+1} \prod_{j \leq s} \sum_{\substack{x_j \in \mathbb{Z} \\ |x_j| \leq B^{3s\eta}}} \frac{(q,x_j)^{1/2}}{x_j^{1/2+(s-1)/s}}. 
\end{align*}
where the inner sum inside the product is bounded by
\begin{align*}
\sum_{\substack{u_1 \leq B^{3s\eta} \\ u_1 \mid q^\infty}} \frac{(q,u_1)^{1/2}}{u_1^{1/2+(s-1)/s}} \sum_{\substack{u_2 \leq B^{3s\eta}/u_1 \\ (u_2,q)=1}} \frac{1}{u_2^{1/2+ (s-1)/s}} \ll q^\varepsilon,
\end{align*}
again by applying Rankin's method and partial summation. This implies that 
\begin{align*}
     \sum_{\substack{\vecx \in \mathbb{Z}^s \\ |\vecx| \leq B^{3s\eta}}} \frac{\sigma_\infty(\vecx)} {|\vecx|^{s-1}} \sum_{q \leq B^{\eta}} q^{-s} S_q(\vecx) \leq \sum_{q \leq B^{\eta}}q^{-s/2+1+\varepsilon} \ll 1,
\end{align*}
which is also sufficient. 

\indent Thus again, one can use M\"obius function to detect residual primitivity and then define
\begin{align*}
    T_0 = T_0(k) = \{\vecx \in (\mathbb{R}_{\neq 0})^s: B^{3s\eta}/k \leq |\vecx| \leq B^{4/(s+2)(s-1)}/k \}.
\end{align*}
Since $\sigma_\infty(k\vecx) = k^{-1} \sigma_\infty(\vecx)$, we establish the following result. 
\begin{lem}
Let $s \geq 7$. Then
\begin{align*}
    M_1(B) = \sum_{q \leq B^\eta} q^{-s} &\sum_{\substack{\veca,\vecb \text{ mod } q\\ (q,\veca) = 1}}c_q(F(\veca;\vecb))\sum_{\substack{k \leq B^{4/(s+2)(s-1)} \\ (k,q) = 1}}     \frac{\mu(k)}{k^s} \sum_{\substack{\vecx \in \mathbb{Z}^s \cap T_0 \\  \vecx \equiv k^{-1}\veca \text{ mod } q}}\frac{\sigma_\infty(\vecx)}{|\vecx|^{s-1}} + O(1).
\end{align*}
\end{lem}
Using \cite[Lemma 6.8]{Browning2}, changing 4 variables to $s$ variables with $s \geq 7$, one immediately deduces the following: 
\begin{lem}
If $\min_j |x_j| \geq 2q$ then
\begin{align*}
    \frac{\sigma_\infty(\vecx)}{|\vecx|^{s-1}} = q^{-s}\int_{[0,q]^s} \frac{\sigma_\infty(\vecx+\vect)}{|\vecx+\vect|^{s-1}} d\vect \;+O(q |\vecx|^{-(s-1)}(\min_j |x_j|)^{-1}|\Delta(\vecx)|^{-1/s}).
\end{align*}
\end{lem}
We can then convert the $\vecx$-sum into an integral.
\begin{lem}
For $s \geq 7$ we have
\begin{align*}
     \sum_{\substack{\vecx \in \mathbb{Z}^s \cap T_0 \\  \vecx \equiv k^{-1}\veca \text{ mod } q}}\frac{\sigma_\infty(\vecx)}{|\vecx|^{s-1}} = q^{-s}J_1(B;k)+O(qk^{s/2}B^{-3\eta(s-1)})
\end{align*} 
where 
\begin{align*}
    J_1(B;k) = \int_{T_0(k)} \frac{\sigma_\infty(\vecy)}{|\vecy|^{s-1}} d\vecy.
\end{align*}
\end{lem}
\begin{proof}
We define
\begin{align*}
    X = \Bigg\{\vecx \in \mathbb{Z}^s: \begin{aligned}
         B^{3s\eta}&/k+2q \leq |\vecx| \leq B^{4/(s+2)(s-1)}/k-2q,\\ &\min |x_j| \geq 2q,
         \vecx \equiv k^{-1}\veca \text{ mod } q
    \end{aligned} \Bigg\},
\end{align*}
and let 
\begin{align*}
    Y = \bigcup_{\vecx \in X}\vecx + [0,q)^s,
\end{align*}
then we note that $Y \subset T_0$ and $T_0 \setminus Y \subset T_1 \cup T_2 \cup T_3$, where 
\begin{align*}
    T_1 &= \{\vect \in T_0: B^{3s\eta}/k \leq |\vect| \leq B^{3s\eta}/k + 3q\}, \\
    T_2 &= \{\vect \in T_0:  B^{4/(s+2)(s-1)}/k-3q \leq |\vect| \leq B^{4/(s+2)(s-1)}/k\}.
\end{align*}
and 
\begin{align*}
    T_3 &= \{\vect \in T_0: \min |t_i| \leq 3q\}.
\end{align*}
Using previous lemma we obtain
\begin{align*}
      \sum_{\substack{\vecx \in \mathbb{Z}^s \cap T_0 \\  \vecx \equiv k^{-1}\veca \text{ mod } q}}\frac{\sigma_\infty(\vecx)}{|\vecx|^{s-1}} = q^{-s}J_1(B;k) + O\left(\sum_{i=0}^3 E_i\right),
\end{align*}
with 
\begin{align*}
    E_0 =   \sum_{\substack{\vecx \in \mathbb{Z}^s \cap T_0 \\  min|x_j| \geq 2q}} q |\vecx|^{-(s-1)}(\min_j |x_j|)^{-1} |\Delta(\vecx)|^{-1/s}
\end{align*}
and 
\begin{align*}
    E_i \ll \sum_{\vecx \in \mathbb{Z}^s \cap T_i} |\vecx|^{-(s-1)}||\Delta(\vecx)|^{-1/2} + q^{-s}\int_{T_i} |\vecy|^{-(s-1)}| |\Delta(\vecy)|^{-1/2} d\vecy
\end{align*}
for $i = 1,2,3$, following from the estimate $|\sigma_\infty(\vecx)| \ll |\vecx|^{-1/2}$ which holds for $s \geq 7$ (c.f. \cite[Page 46]{Davenport}, where $\sigma_\infty(\vecx)$ showed up as the constant term in $\mathcal{N}(P)$ for $s \geq 7$). Note that we've dropped the congruence conditions in the error term. In particular, we have that 
\begin{align*}
    E_0 &\ll q \sum_{\substack{2q \leq x_1 \leq x_2, \dots , x_{s-1} \leq x_s \\ x_s \geq B^{3s\eta}/k}} x_1^{-(s+1)/s} (x_2\dots x_{s-1})^{-1/s} x_s^{-(s-1)-1/s} \\
    &\ll q^{(s-1)/s} \sum_{x_s \geq B^{3s\eta}/k} x_s^{-(s-1)-1/s+(s-1)(s-2)/s} 
    \ll (qkB^{-3s\eta})^{(s-1)/s}.
\end{align*}
Similarly we have that 
\begin{align*}
    E_1 &\ll \sum_{\substack{2q \leq x_1, \dots , x_{s-1} \leq x_s \\B^{3s\eta}/k \leq x_s \leq B^{3s\eta}/k+3q}} x_s^{-(s-1)-1/2}(x_1\dots x_{s-1})^{-1/2} + \int_{B^{3s\eta}/k}^{B^{3s\eta}/k+3q}y_s^{-(s-1)-1/2+(s-1)/2} d y_s \\
    &\ll \sum_{B^{3s\eta}/k \leq x_s \leq B^{3s\eta}/k+3q} x_s^{-(s-1)-1/2+(s-1)/2} + \int_{B^{3s\eta}/k}^{B^{3s\eta}/k+3q}y_s^{-(s-1)-1/2+(s-1)/2} d y_s \\
    &\ll q(kB^{-3s\eta})^{s/2}
\end{align*}
and 
\begin{align*}
    E_2 \ll q(kB^{-4/(s+2)(s-1)})^{s/2}.
\end{align*}
Finally for $E_3$ we evaluate the sum, as the integral is treated similarly. We have that  
\begin{align*}
    E_3 &\ll \sum_{\substack{x_1 \leq 3q \\ 1 \leq x_2,\dots, x_{s-1}\leq x_s \\ x_s \geq B^{3s\eta}/k}} (x_1\dots x_{s-1})^{-1/2}x_s^{-(s-1)-1/2} \\
    &\ll q^{1/2} \sum_{x_s \geq B^{3s\eta}/k} x_s^{-s+1/2+(s-1)/2} \\
    &\ll q^{1/2} (kB^{-3s\eta})^{s/2-1}.
\end{align*}
Combining, we see that 
\begin{align*}
    \sum_{\substack{\vecx \in \mathbb{Z}^s \cap T_0 \\  \vecx \equiv k^{-1}\veca \text{ mod } q}}\frac{\sigma_\infty(\vecx)}{|\vecx|^{s-1}} = q^{-s}J_1(B;k)+O(qk^{s/2}B^{-3\eta(s-1)}),
\end{align*}
which concludes the proof. 
\end{proof}
We can then combine everything to get
\begin{align*}
    M_1(B) = \sum_{q \leq B^{\eta}} q^{-2s} \sum_{\substack{\veca,\vecb \text{ mod } q\\ (q,\veca) = 1}}c_q(F(\veca;\vecb))\sum_{\substack{k \leq B^{4/(s+2)(s-1)} \\ (k,q) = 1}}     \frac{\mu(k)}{k^s} J_1(B;k) +O(1).
\end{align*}
Denote
\begin{align*}
    \psi(q) = \sum_{\substack{\veca,\vecb \text{ mod } q\\ (q,\veca) = 1}}c_q(F(\veca;\vecb)),
\end{align*}
Heath-Brown and Browning in \cite[Lemma 6.10]{Browning2} have established that $\psi$ is multiplicative; changing $4$ variables to $s$ variables when appropriate, one may easily establish the following lemma, using the same argument.
\begin{lem}
$\psi$ is multiplicative, and 
\begin{align*}
    \psi(p^f) = \begin{cases}
    \ \varphi(p^f)p^{3sf/2}(1-p^{-s}) & \text{if }2 \mid f,\\
    \ 0 &  \text{if }2 \nmid f.
    \end{cases}
\end{align*}
for positive integer $f$ and prime $p$.
\end{lem}
We also get an estimate for $J_1(B;k)$, which follows analogously from \cite[Lemma 6.11]{Browning2}.
\begin{lem}
For $s \geq 7$, we have
\begin{align*}
    J_1(B;k) = (s-1) \tau_\infty \log B^{4/(s+2)(s-1)-3s\eta} .
\end{align*}
\end{lem}

Using Lemma 5.10, one sees that
\begin{align*}
    \sum_{\substack{k \leq B^{4/(s+2)(s-1)} \\ (k,q) = 1}} \frac{\mu(k)}{k^s} &J_1(B;k) = (s-1) \tau_\infty \log B^{4/(s+2)(s-1)-3s\eta}\sum_{\substack{k \leq B^{4/(s+2)(s-1)} \\ (k,q) = 1}} \frac{\mu(k)}{k^s}  \\
    &=(s-1) \tau_\infty \log B^{4/(s+2)(s-1)-3s\eta}\sum_{\substack{k=1\\(k,q) = 1}}^\infty \frac{\mu(k)}{k^s} + O(B^{-4/(s+2)}\log B)\\
    &=\frac{(s-1) \tau_\infty \log B^{4/(s+2)(s-1)-3s\eta}}{\zeta(s)}\prod_{p \mid q}(1-p^{-s})^{-1}  + O(B^{-4/(s+2)}\log B).
\end{align*}
Finally, from Lemma 5.9 we see that $\psi$ is only supported on squares with $\psi(q) \leq q^{3s/2+1}$. It follows from Lemma 5.6 and 5.10 that 
\begin{align*}
    M_1(B) = \frac{(s-1) \tau_\infty \log B^{4/(s+2)(s-1)-3s\eta}}{\zeta(s)}\sum_{q=1}^\infty q^{-2s}\psi(q)\prod_{p \mid q}(1-p^{-s})^{-1} +O_\eta(1).
\end{align*}
From Lemma 5.9 we have
\begin{align*}
    \sum_{q=1}^\infty q^{-2s}\psi(q)\prod_{p \mid q}(1-p^{-s})^{-1} &= \prod_p \Big(1+\frac{p^{-4s}\psi(p^2)+p^{-8s}\psi(p^4)+...}{1-p^{-s}}\Big)\\
    &= \prod_p\Big(1+(1-p^{-1})(\sum_{f=1}^\infty p^{(2-s)f})\Big)\\
    &=\prod_p \Big(1+\frac{p-1}{p} \frac{p^{3-s}-p^{2-s}}{p-p^{3-s}}\Big)\\
    &= \prod_p \Big(\frac{1-p^{1-s}}{1-p^{2-s}} \Big)\\
    &=\frac{\zeta(s-2)}{\zeta(s-1)}.
\end{align*}
We can then conclude from Lemma 5.6, 5.8, 5.9 and 5.10 that
\begin{align*}
    M_1(B) = \frac{4\zeta(s-2)}{(s+2)\zeta(s)\zeta(s-1)} \tau_\infty \log B + O(\eta B\log B)+ O_\eta(B).
\end{align*}
Combining with results from \S 4 we see that 
\begin{align*}
    N(\Omega,B) &=\frac{B}{4\zeta(s-2)}M_1(B) + \frac{B}{4\zeta(s-1)}M_2(B)+O(\eta B\log B)+ O_\eta(B)
\end{align*}
which holds for all $\eta \in (0,1/s^4)$. Let the first error term be $E_1 \leq C_1\eta B\log B$, and the second error term be $E_2 \leq c_2(\eta)B$. We wish to show that for any small $\varepsilon >0$, there exists $B(\varepsilon)$ such that $E_i \leq \varepsilon B \log B $, $i = 1,2$, for all $B \geq B(\varepsilon)$. This will then imply that $E_i = o(B \log B)$, which is satisfactory. 

Let $\eta = \min(1/s^4,\varepsilon/C_1)$, then $E_1 \leq \varepsilon B \log B$ for all $B$.  With this choice of $\eta$, set
\begin{align*}
    B(\varepsilon) = e^{c_2(\eta)/\varepsilon},
\end{align*}
then 
\begin{align*}
    c_2(\eta)B \leq \varepsilon B \log B(\varepsilon) \leq \varepsilon B \log B
\end{align*}
when $B \geq B(\varepsilon)$. This proves the claim. 

In particular, it then follows that
\begin{align*}
   N(\Omega,B) &\sim \frac{1}{4\zeta(s)\zeta(s-1)}\tau_\infty B\log B.
\end{align*}
\indent To complete the proof of Theorem 1.1 we need to show that the leading constant agrees with the constant $c_{\textit{Peyre}}$ predicted by Peyre in \cite{Peyre}. According to Schindler \cite[\S 3]{Schindler}, one has
\begin{align*}
    c_{\textit{Peyre}} = &\frac{1}{4\zeta(s-1)\zeta(s-2)}\tau_\infty \prod_p \sigma_p,
\end{align*}
where $\tau_\infty$ is as previously defined in Theorem 1.1 and 
\begin{align*}
    \sigma_p &= \lim_{l \rightarrow \infty} p^{-(2s-1)l}\#\big\{(\vecx,\vecy) \text{ mod }p^l: F(x;y)\equiv 0 \text{ mod } p^l\big\}.
\end{align*}
We note that 
\begin{align*}
    n(p^l) &= \sum_{j=0}^l \sum_{\substack{y \text{ mod } p^l \\ (y,p^l) = p^j}}\#\big\{\vecx \in (\mathbb{Z}/p^l\mathbb{Z})^s: F(x;y)\equiv 0 \text{ mod } p^l\big\}\\
    &= \sum_{j=0}^{[l/2]} \sum_{\substack{(u,p)=1 \\ u \text{ mod } p^{l-j} }}\#\big\{\vecx \in (\mathbb{Z}/p^l\mathbb{Z})^s: F(x;u)\equiv 0 \text{ mod } p^{l-2j}\big\} + O(p^{3/2l}).
\end{align*}
Since $(u,p^j) = 1$, the number of solutions for each $x_i \in \mathbb{Z}/p^{l-2j}\mathbb{Z}$ is $p^{(l-2j)(s-1)}$, and thus 
\begin{align*}
   n(p^l) &=\sum_{j=0}^{[l/2]} \sum_{\substack{(u,p)=1 \\ u \text{ mod } p^{l-j} }} p^{(l-2j)(s-1)}p^{2sj} + O(p^{3s/2l})\\
    &= \sum_{j=0}^{[l/2]}(p^{s(l-j)}-p^{s(l-j)-s})p^{(l-2j)(s-1)}p^{2sj} + O(p^{3s/2l})\\
    &=(1-p^{-s})\sum_{j=0}^{[l/2]}p^{2sl-l+(2-s)j} + O(p^{3s/2l}) \\ 
    &=\frac{1-p^{-s}}{1-p^{s-2}}p^{(2s-1)l}+ O(p^{3s/2l}).
\end{align*}
In particular, taking limit as $l \rightarrow \infty$ , we see that for $s \geq 7$
\begin{align*}
    c = &\frac{1}{4\zeta(s-1)\zeta(s-2)}\tau_\infty \prod_p  \lim_{l \rightarrow \infty} p^{-(2s-1)l}n(p^l)\\
    = &\frac{1}{4\zeta(s-1)\zeta(s)}\tau_\infty.
\end{align*}
\indent Thus our prediction agrees with the constant predicted by Peyre. This proves the claim and Theorem 1.1 follows. 

\appendix
\section{}
We resolve an issue raised in the proof of Lemma 2.2, where we replaced the $L^2$ norm with $L^\infty$ norm in using \cite[Lemma 2]{Schmidt}. 

Let $\Lambda^k$ be lattice with determinant $d(\Lambda^{k})$ in euclidean space $\mathbb{R}^n$ with $n \geq k$, let $\vecx_1,\cdots,\vecx_k$ be vectors in the lattice such that $\|\vecx_1\|_2,\cdots,\|\vecx_k\|_2$ are the successive minima of $\Lambda^k$. Let $\Lambda^{k(|-1)}$ with determinant $d(\Lambda^{k(|-1)})$ be the sublattice of $\Lambda^{k}$ which are spanned by $\vecx_1,\cdots,\vecx_{k-1}$, and let $\Lambda^{k(|-i)}$ with determinant $d(\Lambda^{k(|-i)})$ be the sublattice of $\Lambda^{k(|-i+1)}$ spanned by $\vecx_1,\cdots,\vecx_{k-i}$. For notational convenience let $\Lambda^{k(|-k)}$ be the $0$ lattice and $\Lambda^{k(|-0)} = \Lambda^k$. Finally let $R^k$ be the intersection of the unit ball in $\R^n$ under $L^\infty$ norm with the vector subspace spanned by $\Lambda^k$; let $N(r)$ be the number of lattice points $\vecx$ of $\Lambda^{k}$ which are also in $R^k$ and let $V(k)$ denote the volume of $R^k$. 

Recall that all norms in $\mathbb{R}^n$ are equivalent; it therefore suffices to establish that 
\begin{align*}
    \left|N(r) - \frac{V(k)r^k}{d(\Lambda^k)}\right| \ll_k 1 + \sum_{j=1}^{k-1} \frac{r^{k-j}}{|\vecx_1|...|\vecx_j|}. 
\end{align*}
We will prove the claim through induction on $k$. Note that by resizing $\Lambda^k$ if necessary, we may assume that $r = 1$. 

When $k=1$, $ \left|N(1) - {V(k)}/{d(\Lambda^k)}\right| \leq 2 \ll_k 1$. 

Suppose now that the claim holds for $k-1$; we show that it holds for $k$. 

First assume that $\|\vecx_k\|_2 \leq 1$. Consider the parallelopiped spanned by $\vecx_1, \cdots, \vecx_{k}$, which will have diameter $ \leq \|\vecx_1\|_2 + ... +  \|\vecx_{k}\|_2 \leq k\|\vecx_k\|_2$. In particular, this implies that $\left|N(1) - {V(k)}/{d(\Lambda^k)}\right|$ won't exceed $1/d(\Lambda^k)$ times the volume of the set of points within $k\|\vecx_k\|_2$ of the border of $R^k$. In particular, this volume is bounded above by 
\begin{align*}
     \leq 2k\|\vecx_k\|_2  \text{vol}(\partial((1+k\|\vecx_k\|_2) R^k)) \ll_k \|\vecx_k\|_2.
\end{align*}
This implies that 
\begin{align*}
    \left|N(1) - \frac{V(k)}{d(\Lambda^k)}\right| &\ll_k \frac{\|\vecx_k\|_2}{d(\Lambda^k)} \\
    &\ll_k \frac{\|\vecx_1\|_2...\|\vecx_{k}\|_2}{d(\Lambda^k)d(\Lambda^{k(|-1)})}\\
    &\ll_k \frac{1}{d(\Lambda^{k(|-1)})}\\
    &\ll_k \frac{1}{\|\vecx_1\|_2...\|\vecx_{k-1}\|_2} \ll_k \frac{1}{|\vecx_1|...|\vecx_{k-1}|},
\end{align*}
where the $2$nd, $3$rd and $4$th inequalities come from Minkowski's second theorem. This establishes the claim when $\|\vecx_k\|_2 \leq 1$. 

Suppose now that $\|\vecx_k\|_2 > 1$. Then any point in $R^k \cap \Lambda^k$ will live in $\Lambda^{k(|-1)}$ by our construction. In particular, from inductive hypothesis we see that $N(1) \ll_k 1+ \sum_{j=1}^{k-1} 1/{|\vecx_1|...|\vecx_j|}$. From Minkowski's second theorem we see that since $\|\vecx_k\|_2 > 1$ and
\begin{align*}
    \frac{\|\vecx_k\|_2}{d(\Lambda^k)} &\ll_k \frac{1}{d(\Lambda^{k(|-1)})} \ll_k \frac{1}{|\vecx_1|...|\vecx_{k-1}|}, 
\end{align*}
we have
\begin{align*}
        \frac{1}{d(\Lambda^k)} &\ll_k \frac{1}{|\vecx_1|...|\vecx_{k-1}|}.
\end{align*}
In particular, this implies also that $V(k)/d(\Lambda^k) \ll_k 1+ \sum_{j=1}^{k-1}, 1/{|\vecx_1|...|\vecx_j|}$. As such we have that 
\begin{align*}
     \left|N(1) - \frac{V(k)}{d(\Lambda^k)}\right| \ll_k 1+ \sum_{j=1}^{k-1} \frac{1}{|\vecx_1|...|\vecx_j|},
\end{align*}
which finishes the inductive step and concludes the proof. 
\newpage




\begin{bibdiv}
  \begin{biblist}


\bib{BT}{article}{
AUTHOR={V.V. Batyrev and Y. Tschinkel},
TITLE={Rational points on some Fano cubic bundles}, 
JOURNAL={C. R. Acad. Sci. Paris S{\'e}r. I Math.},
VOLUME={323},
DATE={1996}, 
PAGES={41--46} 
}

\bib{Browning1}{article}{
   AUTHOR={T.D. Browning and L .Q. Hu},
   TITLE={Counting rational points on biquadratic hypersurfaces},
   DATE={2018},
 }

  \bib{Browning2}{article}{
   AUTHOR={T.D. Browning and D.R. Heath-Brown},
   TITLE={Density of Rational Points on a Quadric Bundle in $\mathbb{P}^s \times \mathbb{P}^s $},
   DATE={2020},
  PUBLISHER={Duke Math Journal},
 }
 \bib{Davenport}{book}{
   AUTHOR={H. Davenport},
   TITLE={Analytic Methods for Diophantine Equations and Diophantine Inequalities},
   DATE={2005},
  PUBLISHER={Cambridge Mathematical Library},
 }

 \bib{FMT}{article}{
 AUTHOR={ J. Franke  and Y. I. Manin and Y. Tschinkel}, 
 TITLE={Rational points of bounded height on Fano varieties}, 
 JOURNAL={Inventiones Mathematicae},
 VOLUME ={95}, 
 DATE={1989}, 
 PAGES={421--435}
 }
 
\bib{HB}{article}{
 AUTHOR={D. R. Heath-Brown}, 
 TITLE={A new form of the circle method, and its application to quadratic forms}, 
 JOURNAL={J. Reine Angew. Math.}, 
 VOLUME={481}, 
 DATE={1996}, 
 PAGES={149--206}
 }

 \bib{HB2}{article}{
 AUTHOR={D. R. Heath-Brown}, 
 TITLE={Cubic forms in 14 variables}, 
 JOURNAL={Invent. math. }, 
 VOLUME={170}, 
 DATE={2007}, 
 PAGES={199--230}
 }

 \bib{HeathBrown}{book}{
   AUTHOR={D.R. Heath-Brown},
   TITLE={The density of rational points on curves and surfaces},
   DATE={2002},
     VOLUME={155},
  PUBLISHER={Annals of Mathematics},
 }

  \bib{Iwaniek}{book}{
   AUTHOR={H. Iwaniek},
   TITLE={Topics in Classical Automorphic Forms},
   DATE={1997},
  PUBLISHER={Amer. Math. Soc.},
 }

\bib{LB}{article}{
AUTHOR = {P. Le Boudec},
 TITLE={Density of rational points on a certain smooth bihomogeneous threefold},
JOURNAL={Int. Math. Res. Not. IMRN},
DATE = {2015}, 
VOLUME= {21}, 
PAGES={10703--10715} 
}


\bib{Minkowski}{article}{
AUTHOR={Minkowski, H.},
TITLE={Geometrie der Zahlen},
JOURNAL = {Teubner, Leipzig-Berlin}
DATE ={1896},
PAGES={199--218}
}
 

 \bib{Peyre}{article}{
   AUTHOR={E. Peyre},
   TITLE={Hauteurs et mesures de Tamagawa sur les variétés de Fano},
   DATE={1995},
  VOLUME={79},
  JOURNAL={Duke Math Journal},
 }

\bib{Rob}{article}{
AUTHOR={Robbiani, M.},
TITLE={On the number of rational points of bounded height on smooth bilinear hypersurfaces in biprojective space.},
JOURNAL = {Journal of the London Mathematical Society (2)}, VOLUME={63}, 
DATE ={2001},
PAGES={33--51}
}



  \bib{Schindler}{article}{
   AUTHOR={D. Schindler},
   TITLE={Manin's conjecture for certain biprojective hypersurfaces},
   DATE={2016},
  VOLUME={714},
  JOURNAL={J. reine angew},
 }
 
\bib{Schmidt}{article}{
   AUTHOR={W.M. Schmidt},
   TITLE={Asymptotic formulae for point lattices of bounded determinant and subspaces of bounded height},
   DATE={1968},
  VOLUME={35},
  JOURNAL={Duke Math Journal},
 }

\bib{Spencer}{article}{
 AUTHOR={Spencer, C. V.}, 
 TITLE={The Manin conjecture for $x_0y_0+\ldots+x_sy_s=0$},
 JOURNAL={Journal of Number Theory},
 VOLUME = {129}, 
 DATE={2009},
 PAGES = {1505--1521}
 }

\bib{Thunder}{article}{
 AUTHOR = {Thunder, J. L.},
 TITLE={Asymptotic estimates for rational points of bounded height on flag varieties}, 
 JOURNAL={Compositio Mathematica},
 VOLUME = {88}, 
 DATE ={1993}, 
 PAGES ={155--186}
 }

 \bib{Titchmarsh}{book}{
   AUTHOR={E.C. Titchmarsh and  D.R.Heath-Brown},
   TITLE={The theory of the Riemann zeta-function. 2nd ed.},
   DATE={1986},
   PUBLISHER={Oxford University Press},
 }

\end{biblist}
\end{bibdiv}
\end{document}